\documentclass[12pt, reqno]{amsart}
\usepackage{amsfonts, amsmath, amsthm, amscd, amsfonts, amssymb, graphicx, color}
\usepackage{tikz-cd}
\usepackage[bookmarksnumbered, colorlinks, plainpages]{hyperref}
\newtheorem{theorem}{Theorem}[section]
\newtheorem{lemma}[theorem]{Lemma}
\newtheorem{proposition}[theorem]{Proposition}

\newtheorem{corollary}[theorem]{Corollary}
\newtheorem{definition}[theorem]{Definition}

\newtheorem{remark}[theorem]{Remark}

\usepackage[bookmarksnumbered, colorlinks, plainpages]{hyperref}
\hypersetup{colorlinks=true,linkcolor=blue, anchorcolor=green, citecolor=cyan, urlcolor=red, filecolor=magenta, pdftoolbar=true}

\usepackage{comment}
\excludecomment{mysection}

\begin{document}
	
\title[Jordan Decomposition for WBV-functions in Ordered Normed Spaces]{Jordan Decomposition for WBV-functions in Ordered Normed Spaces}
\author{Amit Kumar}
	
\address{Department of Mathematical and Computational Sciences, National Institute of Technology Karnataka (NITK), Surathkal, Mangaluru- 575025, India.}

\email{\textcolor[rgb]{0.00,0.00,0.84}{amit231291280895@gmail.com} and \textcolor[rgb]{0.00,0.00,0.84}{amit@nitk.edu.in}}

	

\begin{abstract}
In this paper, we define two relations one by orthogonality in vector lattices named as strong relation and the other by bounded linear functionals in normed spaces named as weak relation. It turns out that strong relation is an equivalence relation. We study some of the characterizations of these relations. Given a non-zero element in a normed space, we construct an extensible cone which makes that normed space, an ordered normed space. This extensible cone induces the weak relation in the normed space. Later, we prove a Jordan Decomposition Theorem in a normed space by the weak relation induced by the extensible cone.
\end{abstract}

\thanks{The author is an Assistant Professor in the Department of Mathematical and Computational Sciences, National Institute of Technology Karnataka (NITK), an Institute of National Importance under Ministry of Education, Government of India, Surathkal, Mangaluru- 575025, India.}

\maketitle

\section{Introduction}
In Mathematical Analysis, theory of functions of bounded variation is one of the fascinating theories. In short, functions of bounded variations are known as BV-functions. The major importance of BV-functions is found in Mathematics, Physics and Engineering in defining generalized solutions of of non-linear problems involving functionals, ordinary and partial differential equations. It is more general class than the class of Riemann integrable functions. If $[a,b]$ denotes a closed interval in the real number system $\mathbb{R},$  then the class of BV-functions $BV([a,b])$ is properly contained in the class of Riemann integrable functions $R([a,b]).$ BV-functions are frequently used in Complex Integration Theory. In Complex Integration Theory, BV-functions are known as rectifiable curves. If $f:[a,b]\to \mathbb{C}$ denotes a continuously differentiable function, then $f$ is a BV-function and the total variation of $f$ (also called length of $f$ over $[a,b]$) is given by the formula $V(f)=\displaystyle\int_{a}^{b}\vert f'(t)\vert dt.$ Camille Jordan was the first person who initated the theory of BV-functions. He studied BV-functions of a single variable in $1881$ for the purpose of dealing with the convergence in fourier series \cite{CJFS}. Then onwards, the theory of BV-functions has also been defined and studied in several variables by Lamberto Cesari in 1936 \cite{LSSF}. The notion of functions of bounded variation valued in a normed space is also well-defined  \cite{VS17}. Moreover, there is well-defined notion of a weaker class of functions valued in normed spaces than the class of functions of bounded variation named as functions of weakly bounded variation (see \cite{VS17}, Definition 1.4). In short, functions of weakly bounded variation are called WBV-functions. In this manuscript, WBV-functions are our main interest around which our whole manuscript revolves. However, knowledge of BV-functions is essential for WBV-functions as these functions are defined in the terms of BV-functions.

We would like to highlight the main reason that make BV-functions possible to be defined in real number system. It turns out that the main reason is order completeness property of real number system. In fact, developing the theory of BV-functions, triangle inequality plays a crucial role. Given $x,y\in \mathbb{R}, \vert x+y\vert \leq \vert x\vert +\vert y\vert$ is triangle inequality. Observe that $\vert x\vert=\sup\lbrace x, -x\rbrace=\max\lbrace x, -x\rbrace$ for all $x\in \mathbb{R}.$ It can be easily prove that triangle inequality holds in $\mathbb{R}$ if and only if supremum exists of two elements in $\mathbb{R}$ (that is nothing but consequence of order compleness property in $\mathbb{R}$). Note that order completeness property is also known as Dedekind completeness property. Triangle inequality in complex number system $\mathbb{C}$ also made possible the study of complex BV-functions. For more informations about BV-functions, we refer to see \cite{JBCF, RLWA} and references therein. Following the reason as footprint, many researchers have definded and studied BV-functions and its generalizations in vector lattices and vector semi-lattices. For that, see \cite{CA17, MS90, TP03, VR02, CS11,  HT09,  GM96,   VV90, BV71, IR19, MF73, VS17} and references therein. Vecor lattices are real vector spaces bearing order sturucture with some additional conditions. 

Order sturucture is an integral part of $C^\ast$-algebras. It provides characterizations of a $C^\ast$-algebra in many aspects. The usefulness of order structure can be seen in the work done in \cite{B06, RVK, Kad51, RJ83, Kak, GKP} and references cited therein. Similar concept of order structure has also been defined and studied in  \cite{MA71, CDOB, J72, HHS74, WN73}. Being impressed by the usefulness and importance of order structure theory, many researchers started working in this direction. Karn among such researchers also started working in this direction. In fact, Karn started working in order theoretic aspects of $C^\ast$-algebras. Quality and importance of Karn's work can be seen in \cite{K10, K14, K16, K18, K19}.

In \cite{K18}, Karn introduced and studied weak notion of vector lattices named as absolutely ordered spaces and absolute order unit spaces. The reason such order spaces are weak vector lattices is that, under the condition \cite[Theorem 4.12]{K16}, absolutely ordered spaces become vector lattices and again using the same condition absolute order unit spaces become unital $AM$-spaces. In fact, vector lattices are precisely absolutely orderd spaces satisfing triangle inequality and  unital $AM$-spaces are precisely absolute order unit spaces satisfying triangle inequality. That is why, Karn called ``absolutely ordered spaces" and "absolute order unit spaces" as ``non-commutative vector lattice models". 

In \cite{VS17}, the authors have defined equivalence relation using orthogonality in Hilbert Space. They have constructed an extensible cone in Hilbert Space by Riesz representation theorem. The authors have shown that Jordan decomposition theorem for BV-functions valued in Hilbert Spaces is satisfied with respect to this equivalence relation. Later, the authors have proved some more finer results for BV-functions valued in Hilbert spaces with respect to the defined equivalence relation. 

The notion of othogonality is also well defined in vector lattices or absolutely ordered spaces. Therefore, we also define some equivalence relation in terms of orthogonality. Consequently, generalize this equivalence relation in a normed space. We construct extensible cones in normed spaces using Hahn Banach extension theorem. Later on, we prove a Jordan decomposition theorem for functions of weakly bounded variations (called WBV-functions in short) valued in a normed space is satisfied with respect to this relation. Moreover, we prove some more finer results for WBV-functions. This is the basic idea of our manuscript.

The development of our paper is as follows. In the second section, we recall basic preliminaries required to understand this manuscript. In second section, we recall definitions of vector lattices, absolutely ordered spaces, absolute order unit space and types of orthogonality in these spaces. In the third section, we define strong and weak type equivalence relations with the help of orthogonality in vector lattices or absolutely ordered spaces see (Definitions \ref{6} and \ref{7}). We study some of the basic properties and characterizations of these equivalence relations (see \ref{3}, \ref{4}, \ref{8}, \ref{9} and \ref{11}). In the fourth section, we recall the definition of an extensible cone in ordered normed spaces. By help of Hahn-Banach extension theorem, we also contruction an extensible cone in normed spaces corresponding to a non-zero element in these spaces (Proposition \ref{5}). This construction is a breakthrough for proving a Jordan Decomposition Theorem for WBV-functions valued in normed spaces (Theorem \ref{12}). In the end of the section, we also find some finer versions of the Jordan decomposition Theorem for WBV-functions (see \ref{001}, \ref{002} and \ref{003}).

\section{Preliminaries}
Throughout this manuscript, $\mathbb{X}$ is a real vector space. If a non-empty subset $\mathbb{X}^+$ of $\mathbb{X}$ satisfies the following two conditions: $x+y$ and $\alpha x\in \mathbb{X}^+$ for all $x,y\in \mathbb{X}^+$ and $\alpha\in \mathbb{R}^+\cup \lbrace 0\rbrace,$ then $\mathbb{X}^+$ is called a cone of $\mathbb{X}$ and in this case, $(\mathbb{X}, \mathbb{X}^+)$ is called a \emph{real ordered vector space}. Let $(\mathbb{X}, \leq)$ be a partial ordered space. Put $\mathbb{X}^+=\lbrace x \in \mathbb{X}:x\geq 0\rbrace.$ Define $x \leq y$ if $y - x \in \mathbb{X}^+.$ It is worth to notice that, $\leq $ is unique in the sense of the following properties: (1) $x \le x$ for all $x\in \mathbb{X},$  (2) If $ x \leq y$ and $y \leq z,$ then $x \leq z$ and (3) If $x \leq y$, then $x + z \leq y + z$ and $\alpha x \leq \alpha y$  for all $z \in \mathbb{X}$ and $\alpha\in \mathbb{R}^+.$ The cone $\mathbb{X}^+$ is said to be \emph{proper} if $\mathbb{X}^+ \cap - \mathbb{X}^+ = \{ 0 \}$ and it is said to be \emph{generating} if $\mathbb{X} = \mathbb{X}^+ - \mathbb{X}^+.$ Observe that $\mathbb{X}^+$ is proper if and only if $\leq$ is anti-symmetric. 

Let $(\mathbb{X},\mathbb{X}^+)$ be an ordered vector space and $x,y\in \mathbb{X}.$ Then the order interval $[x,y]$ in $\mathbb{X}$ is defined  by $[x,y]=\lbrace z\in \mathbb{X}:x\leq z\leq y\rbrace.$ If we say that $[x,y]$ is an order interval, then it means that there exists a ordered vector space $(\mathbb{X}, \mathbb{X}^+)$ in which $[x,y]$ is an order interval. 

Let $e\in \mathbb{X}^+.$ Then $e$ is said to be order unit for $\mathbb{X}$ if given $x\in \mathbb{X},$ there exists some $\epsilon > 0$ such that $\epsilon e\pm x\in \mathbb{X}^+.$ The cone $\mathbb{X}^+$ is said to be \emph{Archimedean} if $x\in \mathbb{X}^+$ whenever $x\in \mathbb{X}$ and a fixed $y\in \mathbb{X}^+$ satisfy $\epsilon y+x\in \mathbb{X}^+$ for all $\epsilon > 0.$ 

Let $(\mathbb{X}, \mathbb{X}^+)$ be a real ordered vector space and  $e$ be an order unit for $\mathbb{X}$ such that $\mathbb{X}^+$ is proper and Archimedean. Given $x\in \mathbb{X},$ we define $$\Vert x\Vert := \inf \lbrace \epsilon > 0: \epsilon e \pm x \in \mathbb{X}^+ \rbrace.$$ It turns out that $\Vert \cdot \Vert$ a norm on $\mathbb{X}$ and it is called the norm determined by $e.$ It can be verify that $\mathbb{X}^+$ is closed in this norm and $\Vert x\Vert e\pm x \in \mathbb{X}^+$ holds for all $x\in \mathbb{X}.$ In this case, we call $\mathbb{X}$ an \emph{order unit space} and we denote it by $(\mathbb{X}, e).$

Let $\mathbb{S}$ be a non-empty subset of a real ordered vector space $\mathbb{X}.$ If there exists $z\in \mathbb{X}$ such that $x\leq z$ for all $x\in \mathbb{S},$ then $\mathbb{S}$ is said to be bounded above in $\mathbb{X}$  by $z$ and in this case, $z$ is said to be an  upper bound of $\mathbb{S}.$ Similarly, if there exists $w\in \mathbb{X}$ such that $w\leq x$ for all $x\in \mathbb{S},$ then $\mathbb{S}$ is said to be bounded below in $\mathbb{X}$  by $w$ and in this case, $w$ is said to be a lower bound of $\mathbb{S}.$ If $z$ is an upper bound of $\mathbb{S}$ and $z\leq w$ whenever $w\in \mathbb{X}$ is any other upper bound of $\mathbb{S},$ then $z \in \mathbb{X}$ is called supremum of $\mathbb{S}.$  Similarly, if $w$ is a lower bound of $\mathbb{S}$ and $w\leq z$ whenever $z\in \mathbb{X}$ is any other lower bound of $\mathbb{S},$ then $w \in \mathbb{X}$ is called infimum of $\mathbb{S}.$ Note that supremum and infimum of $\mathbb{S}$ are unique in $\mathbb{X}.$ If $z$ and $w$ denote the supremum and the infimum of $\mathbb{S}$ in $\mathbb{X},$ then we write:  $\sup\lbrace x:x\in \mathbb{S}\rbrace=z$ and $\inf\lbrace x:x\in \mathbb{S}\rbrace=w.$  It is worth to observe that $\sup\lbrace x:x\in \mathbb{S}\rbrace$ exists in $\mathbb{X}$ if and only if $\inf \lbrace -x:x\in \mathbb{S}\rbrace$ exists in $\mathbb{X}$ and in this case, we have $\sup\lbrace x:x\in \mathbb{S}\rbrace=-\inf\lbrace -x:x\in \mathbb{S}\rbrace.$

A \emph{vector lattice} is a real ordered vector space $\mathbb{X}$ in which supremum of any pair $x$ and $y$ denoted by $\sup\lbrace x, y\rbrace$ exists. In a vector lattice $\mathbb{X},$ infimum of any pair $x$ and $y$ also exists. In a vector lattice, we write: $x\vee y=\sup \lbrace x,y\rbrace,x\wedge y=\inf \lbrace x,y\rbrace$ and $\vert x\vert=x\vee (-x).$ Note that $x\vee y=-((-x)\wedge (-y))$ holds in a vector lattice. 

%If supremum of every non-empty bounded above subset of a vector lattice $\mathbb{X}$ exists, then the vector lattice $\mathbb{X}$ is said to be Dedekind complete.

An ordered vector space $(\mathbb{X},\mathbb{X}^+)$ with a norm $\Vert \cdot \Vert$ is said to be ordered normed space if $\mathbb{X}$ is proper. We denote it by $(\mathbb{X},\mathbb{X}^+, \Vert \cdot \Vert).$

An ordered normed $(\mathbb{X}, \mathbb{X}^+,\Vert \cdot \Vert)$ is said to an \emph{$AM$-space} if the following two conditions are satisfied:
\begin{enumerate}
\item[(1)] $\Vert \cdot \Vert$ is a complete norm on $\mathbb{X}.$
\item[(2)] $(\mathbb{X}, \mathbb{X}^+)$ is a vector lattice.
\item[(3)] $\Vert x \vee y\Vert = \textrm{max} \lbrace \Vert x\Vert, \Vert y\Vert \rbrace$ for all $x,y \in \mathbb{X}^+$ such that $x\wedge y=0.$  
\end{enumerate}

For the preliminaries discussed above, we refer to see \cite{MA71, CDOB, J72, HHS74, WN73}.

Now, we recall a possible non-commutative model for vector lattices introduced by Karn named as absolutely ordered spaces \cite{K18}.

\begin{definition} \cite[Definition 3.4]{K18}\label{30}
A real ordered vector space $(\mathbb{X}, \mathbb{X}^+)$ with a mapping $\vert\cdot\vert: \mathbb{X} \to \mathbb{X}^+$ is said to be \emph{absolutely ordered space} denoted by $(\mathbb{X}, \mathbb{X}^+, \vert \cdot \vert)$ if the following conditions are satisfied:               
     \begin{enumerate}
          \item[(a)] $\vert x \vert = x$ if $x \in \mathbb{X}^+.$
          \item[(b)] $\vert x \vert \pm x \in \mathbb{X}^+$ for all $x \in \mathbb{X}.$
          \item[(c)] $\vert \alpha \cdot x \vert = \vert \alpha \vert \cdot \vert x \vert$ for all $x \in \mathbb{X}$ and $\alpha \in \mathbb{R}.$
          \item[(d)] If $x, y$ and $z \in \mathbb{X}$ with $\vert x - y \vert = x + y$ and $0 \leq z \leq y,$ then $\vert x - z \vert = x + z.$
          \item[(e)] If $x, y$ and $z \in \mathbb{X}$ with $\vert x - y \vert = x + y$ and $\vert x - z \vert = x + z,$ then $\vert x - \vert y \pm z \vert \vert = x + \vert y \pm z \vert.$ 
     \end{enumerate}   
\end{definition}  

\begin{remark}\label{14}
Let $\mathbb{X}$ be an absolutely ordered space and $\pm x\in \mathbb{X}^+.$ By \ref{30}(a) and (c), we have $x=\vert x\vert =\vert -x\vert =-x$ so that $2x=0.$ Consequently $x=0$ and hence $\mathbb{X}^+$ is proper.
\end{remark}

Next result shows that absolutely ordered space is a nearer structure to a vector lattice that is reason it is called by Karn as a possible non-commutative model for vector lattices.

\begin{theorem}\label{2}
Given an absolutely ordered space $(\mathbb{X},\mathbb{X}^+,\vert \cdot\vert)$ and $x,y\in \mathbb{X},$ we write: $$x \vee y:=\frac{1}{2}(x+y+\vert x-y\vert).$$ The following set of  statements is equivalent:
\begin{enumerate}
\item[(1)] $x \vee y=\sup \lbrace x,y\rbrace$ for all $x,y\in \mathbb{X}.$
\item[(2)] $\vee$ is associative in $\mathbb{X}.$
\item[(3)] $\pm x\leq z$ implies $\vert x\vert \leq z$ for all $x,z\in \mathbb{X}.$
\item[(4)] $\vert x+y\vert \leq \vert x\vert +\vert y\vert.$
\end{enumerate}
\end{theorem}
     
In the following definition, we recall few types of orthogonalities in absolutely ordered spaces.     
     
\begin{definition}[\cite{K18}, Definition 3.6]\label{10}
	Let $\| \cdot \|$ be a norm defined on an absolutely ordered space $(\mathbb{X}, \mathbb{X}^+, \vert\cdot\vert)$ and $x, y \in \mathbb{X}^+.$
	\begin{enumerate} 
		\item[(a)] $x$ said to be \emph{orthogonal} to $y$ ($x \perp y$) if $\vert x - y \vert = x + y.$ We write: $x^+ := \frac{1}{2}(\vert x \vert + x)$ and $x^- := \frac{1}{2}(\vert x \vert - x).$ Then $x = x^+ - x^-$ and $\vert x \vert = x^+ + x^-,$ therefore $x^+ \perp x^-.$ This decomposition is unique in the sense: $x = x_1 - x_2$ with $x_1 \perp x_2,$ then $x_1 = x^+$ and $x_2 = x^-.$ Thus every element in $\mathbb{X}$ has a unique orthogonal decomposition in $\mathbb{X}^+.$ Moreover, if $x$ and $y\in \mathbb{X},$ then $x$ is said to be \emph{orthogonal} to $y$ (we still write $x \perp y$) if $\vert x\vert \perp \vert y\vert$  (see \cite[Definition 2.2]{PI19}).
		\item[(b)] $x$ is said to be \emph{$\infty$-orthogonal} to $y$ ($x \perp_\infty y$) if $\Vert \alpha x + \beta y\Vert = \max \lbrace \Vert \alpha u \Vert, \Vert \beta v \Vert \rbrace$ for all $\alpha, \beta \in \mathbb{R}.$
		\item[(c)] $x$ is said to be \emph{absolutely $\infty$-orthogonal} to $y$ ($x \perp_\infty^a y$) if $x_1 \perp_\infty y_1$ whenever $0 \leq x_1 \leq x$ and $0 \leq y_1 \leq y.$
	\end{enumerate}
\end{definition}

The next result describes some properties of a vector lattice.

\begin{theorem}[\cite{CDOB}]\label{13}
Let $\mathbb{X}$ be a vector lattice and $x,y,z\in \mathbb{X}.$ Then the following statements hold:
\begin{enumerate}
\item[(1)] $x\vee y=\frac{1}{2}(x+y+\vert x-y\vert).$
\item[(2)] $x\wedge y=\frac{1}{2}(x+y-\vert x-y\vert).$
\item[(3)] $\vert x-y\vert=x\vee y-x\wedge y.$
\item[(4)] $\vert \vert x\vert - \vert y\vert \vert\leq \vert x\pm y\vert$
\item[(5)] $\vert x\vee z-y\vee z\vert \leq \vert x-y\vert$ and $\vert x\wedge z-y\wedge z\vert \leq \vert x-y\vert$ (Birkhoff's inequalities).
\item[(6)] $\vert x^+ -y^+\vert \leq \vert x-y\vert$ and $\vert x^- -y^-\vert \leq \vert x-y\vert.$
\item[(7)] $\vert x+y\vert \vee \vert x-y\vert=\vert x\vert + \vert y\vert.$
\item[(8)] $\vert x+y\vert \wedge \vert x-y\vert=\vert \vert x\vert -\vert y\vert \vert.$
\end{enumerate}
\end{theorem}
 
Finally, we recall definition of absolute order unit spaces which nearer structure to \emph{$AM$-space} \cite{K18}.

\begin{definition}[\cite{K18}, Definition 3.8]
Let $(\mathbb{X}, \mathbb{X}^+, \vert \cdot \vert)$ be an absolutely ordered space with an order unit norm $\Vert \cdot \Vert$ defined on $\mathbb{X}$ determined by the order unit $e$ such that $\mathbb{X}^+$ is norm closed. If $\perp = \perp^a_\infty$ on $\mathbb{X}^+,$ then $(\mathbb{X}, \mathbb{X}^+, \vert \cdot \vert, e)$ is said to an \emph{absolute order unit space}. 
\end{definition}

The self-adjoint part of a unital C$^*$-algebra is an absolute order unit space \cite[Remark 3.9(1)]{K18}. In fact, more generally, every unital $JB$-algebra is also an absolute order unit space.   

Finally, we bind up this section by recalling the definition of functions of weakly bounded variation.

\begin{definition}[\cite{VS17}, Definition 1.4]
Let $\mathbb{X}$ be a normed space and $[a,b]$ be an order interval in any other ordered vector space $(\mathbb{Y}, \mathbb{Y}^+).$ Consider a function $f:[a,b]\to \mathbb{R}.$ Then $f$ is said to be a function of weakly bounded variation if $x^*\circ f:[a,b]\to \mathbb{R}$ is a function of bounded variation for all $x^*\in \mathbb{X}^*,$ where $\mathbb{X}^*$ denotes the normed space of all the bounded linear functionals on $\mathbb{X}.$ In short, functions of weakly bounded variation are called WBV-functions.
\end{definition}

\section{Strong and weak relations}

We begin this section with the definition of relation arising from \ref{10}(a). 

\begin{definition}\label{6}
Let $\mathbb{X}$ be an absolutely ordered space and $x_0\in \mathbb{X}.$ For $x$ and $y\in \mathbb{X},$ we say $x$ is strongly related to $y$ if $x-y\perp x_0.$ We denote strong relation of $x$ with $y$ by $x\sim_{x_0} y.$   
\end{definition}

\begin{remark}
If $x\sim_{x_0} y,$ then $x\pm z\sim_{x_0} y\pm z$ for all $z\in \mathbb{X}.$
\end{remark}

In the next result, it turns out that strong relation $\sim_{x_0}$ is an equivalence relation on a vector lattice $\mathbb{X}.$

\begin{proposition}
Let $\mathbb{X}$ be a vector lattice and $x_0\in \mathbb{X}.$ Then strong relation $\sim_{x_0}$ is an equivalence relation on $\mathbb{X}.$
\end{proposition}

\begin{proof}
Let $x,y$ and $z\in \mathbb{X}.$ Since $x-x=0\perp x_0,$ we get thet $x\sim_{x_0} x.$ Thus $\sim_{x_0}$ is reflexive. Next, assume that $x\sim_{x_0}y$ so that $x-y\perp x_0.$ By Definition \ref{30}(3), we get that $y-x\perp x_0.$ In this case, $y\sim_{x_0}x$ and consequently $\sim_{x_0}$ is symmetric. Finally, also assume that $y\sim_{x_0} z.$ Then $y-z\perp x_0.$ By definition \ref{30}(e), we conclude that $\vert (x-y)\vert +\vert (y-z)\vert \perp \vert x_0\vert.$ Using characterization theorem \ref{2} of vector lattice, we have $\vert x-z\vert \leq \vert (x-y)\vert +\vert (y-z)\vert$ and consequently by definition \ref{30}(d), we get $\vert x-z\vert \perp \vert x_0\vert~i.e.~x-z\perp x_0.$  Thus $x\sim_{x_0}z$ so that $\sim_{x_0}$ is also transitive. Hence $\sim_{x_0}$ is an equivalence relation on $\mathbb{X}.$
\end{proof}

The following result explains that scalar multiplication preserves strong relation.

\begin{proposition}\label{3}
Let  $\mathbb{X}$ be an absolutely ordered space and $x_0\in \mathbb{X}.$ If $x\sim_{x_0} y$, then 
\begin{enumerate}
\item[(1)] $\alpha x\sim_{x_0} \alpha y$ for all $\alpha\in \mathbb{R}.$
\item[(2)] $x\sim_{\beta x_0}y$ for all $\beta \in \mathbb{R}.$
\end{enumerate}
In particular, If $x\sim_{x_0} y$, then $\alpha x\sim_{\beta x_0} \alpha y$ for all $\alpha, \beta \in \mathbb{R}.$
\end{proposition}

\begin{proof} Let $x,y\in \mathbb{X}$ such that $x\sim_{x_0} y$ and $\alpha,\beta \in \mathbb{R}.$ Then $\vert x-y\vert \perp \vert x_0\vert.$ There are two possibilities: either $\vert \alpha\vert \leq 1$ or $\vert \alpha\vert > 1.$ First assume that $\vert \alpha\vert \leq 1.$ Since $\vert \alpha x-\alpha y\vert=\vert \alpha\vert \vert x-y\vert \leq \vert x-y\vert \perp \vert x_0 \vert,$ by definition \ref{30}(d), we get $\vert \alpha x-\alpha y\vert \perp \vert x_0\vert.$ Thus $\alpha x\sim_{x_0} \alpha y.$ Finally, assume that $\vert \alpha\vert >1.$ Then $\vert x-y\vert \perp \vert x_0\vert \geq \frac{1}{\vert \alpha \vert}\vert x_0\vert,$ again by definition \ref{30}(d), we get  $x-y \perp \frac{1}{\alpha}x_0.$ Therefore $\alpha x-\alpha y \perp x_0~i.e.~\alpha x \sim_{x_0}\alpha y.$ 
Similarily following the proof of (1), we can also show that $x\sim_{\beta x_0}y.$

Next, if $x\sim_{x_0} y,$ then $\alpha x\sim_{x_0} \alpha y$ follows from the part (1) and now applying the part (2) for $\alpha x\sim_{x_0} \alpha y,$ we conclude that $\alpha x\sim_{\beta x_0} \alpha y.$
\end{proof}

The next result shows that the strong relation $\sim_{x_0}$ is additive.

\begin{lemma}\label{4}
Let $\mathbb{X}$ be a vector lattice and $x_0,x_1, x_2, y_1, y_2 \in \mathbb{X}.$ If $x_1\sim_{x_0} y_1$ and  $x_2\sim_{x_0} y_2,$ then $x_1\pm x_2\sim_{x_0}y_1\pm y_2.$
\end{lemma}

\begin{proof}
Let $x_1\sim_{x_0} y_1$ and  $x_2\sim_{x_0} y_2$ in $\mathbb{X}.$ Then $\vert x_1-y_1\vert \perp \vert x_0\vert$ and $\vert x_2-y_2\vert \perp \vert x_0\vert$ so that $\vert x_1-y_1\vert +\vert x_2-y_2\vert \perp \vert x_0\vert$ by definition \ref{30}(e). Since $\mathbb{X}$ is a vector lattice, we have $\vert (x_1+x_2)-(y_1+y_2)\vert \leq \vert x_1-y_1\vert +\vert x_2-y_2\vert.$  Thus $\vert (x_1+x_2)-(y_1+y_2)\vert \perp \vert x_0\vert~i.e.~x_1+x_2\sim_{x_0} y_1+y_2$ follows from by definition \ref{30}(d). 

Since $x_2\sim_{x_0} y_2$ in $\mathbb{X},$ by Proposition \ref{3}(1), we get that $-x_2\sim_{x_0} -y_2$ in $\mathbb{X}.$ Following the same proof, $x_1\sim_{x_0} y_1$ and  $-x_2\sim_{x_0} -y_2$ implies $x_1- x_2\sim_{x_0}y_1- y_2.$
\end{proof}

In the next result, we characterize strong equivalence $\sim_{x_0}$ in terms of positive and negative parts, supremum and infimum.

\begin{proposition}\label{8}
Let $\mathbb{X}$ be a vector lattice and $x_0,x$ and $y \in \mathbb{X}.$ Then the following statements are equivalent:
\begin{enumerate}
\item[(1)] $x\sim_{x_0}y$ 
\item[(2)] $x^+\sim_{x_0}y^+$ and $x^-\sim_{x_0}y^-$
\item[(3)] $x\vee y \sim_{x_0} x\wedge y$
\item[(4)] $x\vee z \sim_{x_0} y\vee z$ and $x\wedge z \sim_{x_0} y\wedge z$ for all $z\in \mathbb{X}$
\end{enumerate}
\end{proposition}

\begin{proof}
(1) and (2) are equivalent: First assume that $x\sim_{x_0}y$ so that $\vert x-y\vert \perp \vert x_0\vert.$ Since $\mathbb{X}$ is a vector lattice, by Theorem \ref{13}(6), we have $\vert x^+-y^+\vert \leq \vert x-y\vert$ and $\vert x^--y^-\vert \leq \vert x-y\vert.$ Therefore $\vert x^+-y^+\vert \perp \vert x_0\vert$ and $\vert x^--y^-\vert \perp \vert x_0\vert$ follows from by definition \ref{30}(d) and consequently $x^+\sim_{x_0}y^+$ and $x^-\sim_{x_0}y^-.$ Conversely, assume that $x^+\sim_{x_0}y^+$ and $x^-\sim_{x_0}y^-.$ By Lemma \ref{4}, we get that $x=x^+-x^-\sim_{x_0} y^+-y^-=y.$ \\
Next, (1) and (3) are equivalent follows from the fact that $\vert x-y\vert = x\vee y-x\wedge y$ by Theorem \ref{13}(3).\\
(1) implies (4): Assume that (1) is true that is $x\sim_{x_0}y~i.e.~\vert x-y\vert \perp \vert x_0\vert.$ Since $\mathbb{X}$ is a vector lattice, by Theorem \ref{13}(5), we have $\vert x\vee z - y\vee z\vert \leq \vert x-y\vert$ and $\vert x\wedge z - y\wedge z\vert \leq \vert x-y\vert$ for all $z\in \mathbb{X}.$ By definition \ref{30}(d), we get $\vert x\vee z - y\vee z\vert \perp \vert x_0\vert$ and $\vert x\wedge z - y\wedge z\vert \perp \vert x_0\vert$ so that $x\vee z \sim_{x_0} y\vee z$ and $x\wedge z \sim_{x_0} y\wedge z$ for all $z\in \mathbb{X}.$\\
(4) implies (2): Assume that (4) is true that is $x\vee z \sim_{x_0} y\vee z$ and $x\wedge z \sim_{x_0} y\wedge z$ for all $z\in \mathbb{X}.$ Put $z=0,$ we get $x\vee 0 \sim_{x_0} y\vee 0$ and $x\wedge 0 \sim_{x_0} y\wedge 0.$ Therefore $x^+\sim_{x_0}y^+$ and $x^-\sim_{x_0}y^-.$
\end{proof}

The following result describes some more properties of strong relation $\sim_{x_0}.$

\begin{proposition}
Let $\mathbb{X}$ be a vector lattice and $x_0,x$ and $y \in \mathbb{X}.$ 
\begin{enumerate}
\item[(1)] $x\sim_{x_0}y$ implies $\vert x\vert \sim_{x_0} \vert y\vert.$
\item[(2)] $\vert x\vert \sim_{x_0}\vert y\vert$ implies $\vert x+y\vert \wedge \vert x-y\vert\sim_{x_0}0.$ 
\item[(3)] $\vert x\vert \sim_{x_0}-\vert y\vert$ implies $\vert x+y\vert \vee \vert x-y\vert\sim_{x_0}0$ and hence $x\sim_{x_0}y$ and $x\sim_{x_0}-y.$
\item[(4)] $x\vee y\sim_{x_0} 0$ and $x\wedge y\sim_{x_0} 0$ implies $x\sim_{x_0}y.$
\item[(5)] $x\sim_{x_0}y$ and $x\sim_{x_0}-y$ implies $x, x\vee y, x\wedge y,\vert x\vert \vee \vert y\vert$ and $\vert x\vert \wedge \vert y\vert \sim_{x_0}0.$ 
\end{enumerate}
\end{proposition}

\begin{proof} 
Let $x_0,x$ and $y \in \mathbb{X}.$ 
\begin{enumerate}
\item[(1)] Assume that $x\sim_{x_0}y$ so that $\vert x-y\vert \perp \vert x_0\vert.$ Since $\mathbb{X}$ is a vector lattice, by Theorem \ref{13}(4), we get that $\vert \vert x\vert-\vert y\vert\vert \leq \vert x-y\vert.$ Therefore $\vert \vert x\vert -\vert y\vert\vert \perp \vert x_0\vert$ so that $\vert x\vert \sim_{x_0} \vert y\vert.$
\item[(2)]  Assume that $\vert x\vert \sim_{x_0}\vert y\vert$ so that $\vert \vert x\vert -\vert y\vert \vert \perp \vert x_0\vert.$ As $\mathbb{X}$ is a vector lattice, again by Theorem \ref{13}(8), we have $\vert x+y\vert \wedge \vert x-y\vert = \vert \vert x\vert -\vert y\vert \vert$ so that $\vert x+y\vert \wedge \vert x-y\vert\sim_{x_0}0.$
\item[(3)] In a vector lattice $\mathbb{X},$ we have $\vert x+y\vert \vee \vert x-y\vert = \vert x\vert + \vert y\vert$ by Theorem \ref{13}(7). Then $\vert x\vert \sim_{x_0}-\vert y\vert$ implies $\vert x\vert + \vert y\vert \perp \vert{x_0}\vert$ so that $\vert x+y\vert \vee \vert x-y\vert \perp \vert x_0\vert~i.e.~\vert x+y\vert \vee \vert x-y\vert\sim_{x_0}0.$ Next, $\vert x-y\vert, \vert x+y\vert \leq \vert x+y\vert \vee \vert x-y\vert$ implies  $\vert x-y\vert \perp \vert x_0\vert$ and $\vert x+y\vert \perp \vert x_0\vert$ so that $x\sim_{x_0}y$ and $x\sim_{x_0}-y$ follows from definition \ref{30}(d).
\item[(4)] Since $\mathbb{X}$ is a vector lattice, applying Theorem \ref{13}(3), we have $\vert x-y\vert=x\vee y-x\wedge y.$ By  Lemma \ref{4}, $x\vee y\sim_{x_0} 0$ and $x\wedge y\sim_{x_0} 0$ implies $\vert x-y\vert=x\vee y-x\wedge y \sim_{x_0} 0~i.e.~x\sim_{x_0}y.$
\item[(5)] Let $x\sim_{x_0}y$ and $x\sim_{x_0}-y$ holds in $\mathbb{X}.$ By Lemma \ref{4}, we have $2x=x+x\sim_{x_0}y-y=0$ and by Proposition \ref{3}(3), we conclude that $x\sim_{x_0}0.$ In fact $x\sim_{x_0}y$ and $x\sim_{x_0}-y$ implies that $\vert x-y\vert \perp \vert x_0\vert$ and $\vert x+y\vert \perp \vert x_0\vert.$ In a vector lattice, by Theorem \ref{13}, the following results hold: $x\vee y=\frac{1}{2}(x+y+\vert x-y\vert), x\wedge y=\frac{1}{2}(x+y-\vert x-y\vert), \vert x\vert \vee \vert y\vert =\frac{1}{2}(\vert x+y\vert +\vert x-y\vert)$ and $\vert x\vert \wedge \vert y\vert =\frac{1}{2}(\vert \vert x+y\vert -\vert x-y\vert \vert).$ Using definition \ref{30}(e), we have $\vert \vert x+y\vert \pm \vert x-y\vert \vert \perp \vert x_0\vert$  so that $\vert x\vert \vee \vert y\vert \perp \vert x_0\vert$ and $\vert x\vert \wedge \vert y\vert \perp \vert x_0\vert$ and consequently $\vert x\vert \vee \vert y\vert \sim_{x_0} 0$ and $\vert x\vert \wedge \vert y\vert \sim_{x_0} 0.$ Since by Theorem \ref{2}, the Triangle inequality holds in a vector lattice, therefore we get that $\vert x\vee y\vert \leq \vert x\vert \vee \vert y\vert$ and $\vert x\wedge y\vert \leq \vert x\vert \wedge \vert y\vert$ so that $\vert x \vee y\vert \perp \vert x_0\vert$ and $\vert x \wedge y\vert \perp \vert x_0\vert.$ Thus $\vert x\vee y\vert \sim_{x_0}0$ and $\vert x\wedge y\vert \sim_{x_0}0.$
\end{enumerate}
\end{proof}

In an absolutely ordered space $\mathbb{X},$ we characterized all $x$ that are strongly related to 0 by $x$ itself in the following result.

\begin{corollary}\label{14}
Let  $\mathbb{X}$ be an absolutely ordered space and $x \in \mathbb{X}.$ Then $x=0$ if and only if $x\sim_x 0 .$
\end{corollary}

\begin{proof}
First assume that $x\sim_x 0.$ Then $\vert x\vert \perp \vert x\vert.$ Thus $2\vert x\vert = \vert x\vert +\vert x\vert =\vert \vert x\vert -\vert x\vert \vert=0$ so that $\vert x\vert =0.$ By definition \ref{30}(b), $\vert x\vert \pm x\in \mathbb{X}^+$ implies that $\pm x\in \mathbb{X}^+.$ Since $\mathbb{X}^+$ is proper by \ref{14}, we have $x=0.$ Next, converse part follows trivially.
\end{proof}

Next, we prove a characterization of projections in absolute order unit spaces in terms of $\sim_{x_0}$ which follows trivially. For the difinition of order projection, we refer to see \cite[Definition 5.2]{K18}.

\begin{corollary}
Let $\mathbb{X}$ be an absolute order unit space and $p\in \mathbb{X}^+$ such that $\Vert p\Vert\leq 1.$ Then the following statements are equivalent:
\begin{enumerate}
\item[(1)] $p$ is an order projection.
\item[(2)] $e-p\sim_{p} 0.$
\item[(3)] $p\sim_{e-p} 0.$
\end{enumerate}
\end{corollary}

Given an ordered normed space $(\mathbb{X}, \Vert \cdot\Vert)$, we denote the set of all the bounded linear functionals on $X^*.$  Now, we define a weak relation for the functions valued in an ordered normed space through a functional $x^*\in X^*.$

\begin{definition}\label{7}
Let $\mathbb{X}$ be a normed space and $f_1,f_2:[a,b]\to \mathbb{X}$ be two functions, and $x^*\in \mathbb{X}^*\setminus \lbrace 0 \rbrace.$ We say that $f_1$ is weakly related to $f_2$ by $x^*$  if $x^*\circ(f_1-f_2)=0~i.e.~x^*(f_1(t))=x^*(f_2(t))$ for all $t\in [a,b].$ If $f_1$ and $f_2$ are related weakly by $x^*,$ we denote it by $f_1 \sim_{x^*} f_2.$
\end{definition}

Note that $f_1 \sim_{x^*} f_2$ implies $f_2 \sim_{x^*} f_1.$ Therefore instead of saying that $f_1$ is weakly related to $f_2$ by $x^*,$ we say that $f_1$ and $f_2$ are weakly related by $x^*.$

In a normed space $\mathbb{X}$ and $\mathbb{A}\subseteq \mathbb{X},$ we denote the closure of $\mathbb{A}$ in $\mathbb{X}$ by $\overline{\mathbb{A}}$ and the linear space of  $\mathbb{A}$ in $\mathbb{X}$ by $Span (\mathbb{A}).$

The weak relation $\sim_{x^*}$ is related to the strong relation $\sim_{x_0}$ by the following characterization for $\sim_{x^*}.$  

\begin{proposition}\label{15}
Let $\mathbb{X}$ be a normed space and $f_1,f_2:[a,b]\to \mathbb{X}$ be two functions, and $x^*\in \mathbb{X}^*\setminus \lbrace 0\rbrace.$ Then the following statements are equivalent:
\begin{enumerate}
\item[(1)] $f_1 \sim_{x^*} f_2$ 
\item[(2)] $x^*\circ (f_1-f_2)(t)\sim_{x^*\circ (f_1-f_2)(t)} 0$ for all $t\in [a,b].$
\item[(3)] Range$(f_1-f_2)\subseteq Null(x^*).$ 
\item[(4)] $Span (Range(f_1-f_2))\subseteq Null(x^*).$
\item[(5)] $\overline{Span (Range(f_1-f_2))}\subseteq Null(x^*).$
\end{enumerate}
\end{proposition}

\begin{proof}
(1) and (2) are equivalent: $f_1 \sim_{x^*} f_2$ if and only if $x^*\circ (f_1-f_2)(t)=0$ for all $t\in [a,b].$ In the light of the Corollary \ref{14}, we have $x^*\circ (f_1-f_2)(t)=0$ if and only if $x^*\circ (f_1-f_2)(t)\sim_{x^*\circ (f_1-f_2)(t)} 0.$ Therefore $f_1 \sim_{x^*} f_2$ if and only if $x^*\circ (f_1-f_2)(t)\sim_{x^*\circ (f_1-f_2)(t)} 0$ for all $t\in [a,b].$

(1) and (3) are equivalent: $f_1 \sim_{x^*} f_2$ if and only if $x^*\circ(f_1-f_2)=0~i.e.~x^*((f_1-f_2)(t))$ for all $t\in [a,b]$ if and only if $\lbrace (f_1-f_2)(t):t\in [0,1]\rbrace\subseteq Null(x^*)$ if and only if $Range(f_1-f_2)\subseteq Null(x^*).$ 

(3) and (4) are equivalent: Assume that (3) is true. Let $y\in Span (Range(f_1-f_2)).$ There exist $x_i\in Range(f_1-f_2)$ and $\delta_i\in \mathbb{R}$ for $i=1,2,\cdots, n$ such that $y=\displaystyle \sum_{i=1}^n \delta_i x_i.$ Since $x_i\in Null(x^*),$ we get that $x^*(x_i)=0$ for all $i.$ Then $x^*(y)=\displaystyle\sum_{i=1}^n \delta_i x^*(x_i)=0$ so that $y\in Null(x^*).$ Thus $Span (Range(f_1-f_2))\subseteq Null(x^*)$ and consequently (4) follows. Conversely, (4) implies (3) is trivial to verify.  

(4) and (5) are equivalent: Assume that (4) holds. For bounded linear functional $x^*, $ the space $Null(x^*)$ is always norm-closed. Therefore $\overline{Null(x^*)}=Null(x^*).$ Then $Span (Range(f_1-f_2))\subseteq Null(x^*)$ implies $\overline{Span (Range(f_1-f_2))}\subseteq Null(x^*).$ Thus $\overline{Span (Range(f_1-f_2))}\subseteq Null(x^*)$ and consequently (5) follows. Conversely, (5) implies (4) is followed immediately from the fact $Range(f_1-f_2)\subseteq \overline{Range(f_1-f_2)}.$
\end{proof}

\begin{corollary}\label{16}
Let $\mathbb{X}$ be a normed space and $f_1,f_2:[a,b]\to \mathbb{X}$ be two functions, and $x^*\in \mathbb{X}^*\setminus \lbrace 0\rbrace.$ If $f_1 \sim_{x^*} f_2,$ then $\overline{Span(Range(f_1-f_2))}\neq \mathbb{X}.$ Conversely, if  $\overline{Span(Range(f_1-f_2))}\neq \mathbb{X},$ then there exists $y^*\in \mathbb{X}^*\setminus \lbrace 0\rbrace$ such that $f_1 \sim_{y^*} f_2.$
\end{corollary}

\begin{proof}
First, assume that $f_1 \sim_{x^*} f_2.$ Using the Proposition \ref{15}, we get that $\overline{Span (Range(f_1-f_2))}\subseteq Null(x^*).$ Since $x^*\neq 0,$ we conclude that $Null (x^*)\neq \mathbb{X}.$ Thus $\overline{Span(Range(f_1-f_2))}\neq \mathbb{X}.$ Conversely, assume that $\overline{Span(Range(f_1-f_2))}$ $\neq \mathbb{X}.$ Put $\mathbb{Y}=\overline{Span(Range(f_1-f_2))}\neq \mathbb{X}.$ Then $\mathbb{Y}$ is a closed norm subspace of $\mathbb{X}.$ As $\mathbb{Y}\neq \mathbb{X},$ by an application of Hahn-Banach extension theorem, there exists $y^*\in \mathbb{X}^*$ such that $\Vert y^*\Vert=1$ and $y^*(y)=0$ for all $y\in \mathbb{Y}.$ Since $(f_1-f_2)(t)\in \mathbb{Y},$ we get $y^*((f_1-f_2)(t))=$ for all $t\in [a,b].$ Thus $y^*\circ (f_1-f_2)=0.$ Hence $f_1 \sim_{y^*} f_2.$ 
\end{proof}

\begin{corollary}
Let $\mathbb{X}$ be a normed space and $f_1,f_2:[a,b]\to \mathbb{X}$ be two functions. If $Span(Range(f_1-f_2))$ is dense in $\mathbb{X},$ then $f_1$ is not weakly related to $f_2$ by any $x^*\in \mathbb{X}^*\setminus \lbrace 0\rbrace.$
\end{corollary}

\begin{proof}
Let $Span(Range(f_1-f_2))$ be dense in $\mathbb{X}.$ Then $\overline{Span(Range(f_1-f_2))}=\mathbb{X}.$ If possible suppose that $f_1\sim_{x^*} f_2$ for some $x^*\setminus \lbrace 0\rbrace.$ By Corollary \ref{16}, we get that $\overline{Span(Range(f_1-f_2))}\neq \mathbb{X}$ which is a contradiction. Hence $f_1$ is not weakly related to $f_2$ by any $x^*\in \mathbb{X}^*\setminus \lbrace 0\rbrace.$
\end{proof}

The next result provides a characterization for the functions $f_1$ and $f_2$ weakly related by every bounded linear functional on  $\mathbb{X}~i.e.~ f_1\sim_{x^*}f_2$ for all $x^*\in \mathbb{X}^*\setminus \lbrace 0\rbrace.$ 

\begin{proposition}\label{9}
Let $\mathbb{X}$ be a normed space and $f_1,f_2:[a,b]\to \mathbb{X}$ be two functions. If $f_1\sim_{x^*}f_2$ for all $x^*\in X^*\setminus \lbrace 0\rbrace,$ then $f_1=f_2.$
\end{proposition}

\begin{proof}
 Let $t\in [a,b].$ Assume that $f_1\sim_{x^*}f_2$ for all $x^*\in X^*\setminus \lbrace 0\rbrace$ so that $x^*(f_1(t)-f_2(t))=0$ for all $x^*\in X^*.$ Then $\Vert f_1(t)-f_2(t)\Vert=\sup \lbrace \vert x^*(f_1(t)-f_2(t))\vert :x^*\in X^*, \Vert x^*\Vert = 1\rbrace=0.$ Thus $f_1(t)=f_2(t)$ for all $t\in [a,b]$ and consequently $f_1=f_2.$   
\end{proof}

For fix $x^*\in X^*,$ the weak relation $\sim_{x^*}$ is an equivalence relation on the set of all the functions defined on the closed interval $[a,b].$ 

\begin{proposition}\label{11}
Let $\mathbb{X}$ be a normed space and $f_1,f_2:[a,b]\to \mathbb{X}$ be two functions. Put $\mathcal{F}=\lbrace f:[a,b]\to \mathbb{X}~\textit{is a function}\rbrace.$ Then for a fixed $x^*\in X^*\setminus \lbrace 0\rbrace,$ the weak relation $\sim_{x^*}$ is an equivalence relation on $\mathcal{F}.$
\end{proposition}

\begin{proof}
Reflexivity and symmetricity of $\sim_{x^*}$ on $\mathcal{F}$ are trivial to verify. Let $f_1,f_2$ and $f_3\in \mathcal{F}$ such that $f_1\sim_{x^*}f_2$ and $f_2\sim_{x^*}f_3.$ By $f_1\sim_{x^*}f_2$ and $f_2\sim_{x^*}f_3,$ we get that $x^*(f_1(t))=x^*(f_2(t))$ and $x^*(f_2(t))=x^*(f_3(t))$ respectively for all $t\in [a,b].$ Then $x^*(f_1(t))=x^*(f_3(t))$ for all $t\in [a,b]$ so that $f_1\sim_{x^*}f_3.$ Thus $\sim_{x^*}$ is a transitive relation on $\mathcal{F}.$ Hence $\sim_{x^*}$ is an equivalence relation on $\mathcal{F}.$
\end{proof}

\section{Extensible cones and Weak type Jordan decomposition theorem}

We begin this section by recalling the definition of a extensible cone in an ordered normed space. In fact, we are preparing a background tor prove a Jordan decomposition theorem for functions of weakly bounded variation. 

\begin{definition}
Let $(\mathbb{X}, \Vert \cdot \Vert)$ be a normed space and $\mathbb{X}^+$ be a cone in $\mathbb{X}.$ Then $\mathbb{X}^+$ is called an extensible cone if there exists a cone $\mathbb{X}_+$ in $\mathbb{X}$ and $\alpha>0$ such that $B(x,\alpha \Vert x\Vert) \subseteq  \mathbb{X}_+$ for any $x\in \mathbb{X}^+.$
\end{definition}

\begin{remark}
If $\mathbb{X}^+$ is an extensible cone in a normed space $(\mathbb{X}, \Vert \cdot \Vert),$ then $\mathbb{X}^+$ is norm-closed. To see this, let $\lbrace x_n \rbrace$ be a sequence in $\mathbb{X}^+$ such that $x_n \to y$ in $\mathbb{X}.$  Then $x^*(x_n)\geq \alpha \Vert x_n\Vert$ for all $n\in \mathbb{N}.$ Since $x^*(x_n)\to x^*(y), \Vert x_n\Vert \to \Vert y\Vert,$ letting $n\to \infty$ in $x^*(x_n)\geq \alpha \Vert x_n\Vert,$ we get that $x^*(y)\geq \alpha \Vert y\Vert.$ Thus $y\in \mathbb{X}^+$ so that $\mathbb{X}^+$ is norm-closed.   
\end{remark}

Next, we recall a characterization for a cone to be an extensible cone in terms of bounded linear functionals. This characterization is very crucial for proving subsequent results.

\begin{theorem}\label{k}
Let $(\mathbb{X}, \Vert \cdot \Vert)$ be a normed space and $\mathbb{X}^+$ be a cone in $\mathbb{X}.$ Then $\mathbb{X}^+$ is extensible if and only if there exist $x^*\in \mathbb{X}^*$ and a constant $\alpha >0$ such that $x^*(x)\geq \alpha \Vert x\Vert$ for all $x\in \mathbb{X}^+.$    
\end{theorem}

In the following result, by Theorem \ref{k}, given a normed space $\mathbb{X}$ and $x_0\in \mathbb{X}\setminus \lbrace 0\rbrace,$ we construct an extensible cone $\mathbb{X}^+$ in $\mathbb{X}$ containing $x_0.$

\begin{proposition}\label{5}
Let $(\mathbb{X}, \Vert \cdot \Vert)$ be a normed space and $x_0 \in \mathbb{X}\setminus\lbrace 0\rbrace.$ Then there exists an extensible cone $\mathbb{X}_{x_0}^+$ containing $x_0$ such that $(\mathbb{X}, \mathbb{X}_{x_0}^+, \Vert \cdot \Vert)$ is an ordered normed space. 
\end{proposition}

\begin{proof}
Let $x_0\in \mathbb{X}\setminus \lbrace 0\rbrace.$ Define $\mathbb{Y}=span\lbrace x_0\rbrace=\lbrace \beta x_0 : \beta \in \mathbb{R}\rbrace.$ Define $y^*(\beta x_0)=\beta \Vert x_0\Vert^2.$ Then $y^*$ is a linear functional on $\mathbb{Y}$ and $y^*(x_0)=\Vert x_0\Vert^2.$ Also, we have $\vert y^*(\beta x_0)\vert =\vert\beta\vert \Vert x_0\Vert^2=\Vert x_0\Vert \Vert \beta x_0\Vert$ so that $y^*\in \mathbb{Y}^*$ with $\Vert y^*\Vert=\Vert x_0\Vert.$ Consider Hahn-Banach extension $x^*$ of $y^*$ to $\mathbb{X}.$ Next, let $\mathbb{X}_{x_0}^+ = \lbrace x\in \mathbb{X}: x^*(x)\geq \alpha \Vert x\Vert\rbrace$ for some $\alpha\in (0,\Vert x_0\Vert].$ Let $x,y\in \mathbb{X}_{x_0}^+$ and $\delta \in \mathbb{R}$ such that $\delta \geq 0.$ Then $x^*(x)\geq \alpha \Vert x\Vert$ and $x^*(y)\geq \alpha \Vert y\Vert.$ Thus $x^*(x+y)\geq \alpha (\Vert x\Vert +\Vert y\Vert)\geq \alpha \Vert x+y\Vert$ and  $g(\delta x)=\delta g(x)\geq \delta \alpha \Vert x\Vert=\alpha \Vert \delta x\Vert$ so that $x+y$ and $\delta x \in \mathbb{X}_{x_0}^+.$ Therefore $\mathbb{X}_{x_0}^+$ is a cone. By Theorem \ref{k}, it follows that $\mathbb{X}_{x_0}^+$  is an extensible cone. Note that $x^*(x_0)=\Vert x_0\Vert^2=\Vert x_0\Vert \Vert x_0\Vert\geq \alpha \Vert x_0\Vert$ so that $x_0\in \mathbb{X}_{x_0^+}.$ Next, assume that $x\in \mathbb{X}_{x_0}^+\bigcap (-\mathbb{X}_{x_0}^+).$ Then $\pm x \in \mathbb{X}_{x_0}^+$ and consequently $\pm x^*(x)=x^*(\pm)\geq \alpha \Vert \pm x\Vert=\alpha \Vert x\Vert \geq 0$ so that $x^*(x)=0.$ Using the fact $\pm x^*(x)=x^*(\pm)\geq \alpha \Vert x\Vert \geq 0,$ we get that $0\geq \alpha \Vert x\Vert \geq 0.$ Since $\alpha >0,$ we get $\Vert x\Vert =0$ that means $x=0.$ Therefore $\mathbb{X}_{x_0}^+$ is proper. Hence $(\mathbb{X}, \mathbb{X}_{x_0}^+, \Vert \cdot \Vert)$ is an ordered normed space. 
\end{proof}

\begin{corollary}
Let $(\mathbb{X}, \Vert \cdot \Vert)$ be a normed space and $x_0 \in \mathbb{X}\setminus\lbrace 0\rbrace.$ 
\begin{enumerate}
\item[(1)] If $\mathbb{X}^+=\lbrace \delta  x_0:\delta \in \mathbb{R}, \delta \geq 0\rbrace,$ then $(\mathbb{X}, \mathbb{X}^+, \Vert \cdot \Vert)$ is an ordered normed space. Moreover, $\mathbb{X}^+$ is norm-closed.
\item[(2)] If $\alpha=\Vert x_0\Vert$ and $x,y\in \mathbb{X}$ such that $0\leq x\leq y$ in $\mathbb{X}_{x_0}^+,$  then $\Vert x\Vert \leq \Vert y\Vert.$
\end{enumerate}
\end{corollary}

\begin{proof}
Consider $x^*\in \mathbb{X}^*$ as induced in the Proposition \ref{5}.
\begin{enumerate}
\item[(1)] Let $\mathbb{X}^+=\lbrace \delta  x_0:\delta \in \mathbb{R}, \delta \geq 0\rbrace.$ Then $\mathbb{X}^+\subseteq \mathbb{X}_{x_0}^+$ and $\mathbb{X}_{x_0}^+$ is cone implies that $\mathbb{X}^+$ is also cone. Note that $0\leq \delta_1 x_0 \leq \delta_2 x_0$ if and only if $0\leq \delta_1 \leq \delta_2.$ Next, let $x,y\in \mathbb{X}^+$ such that $0\leq x\leq y.$ There exist $\delta_1, \delta_2\in \mathbb{R}$ such that $0\leq \delta_1 \leq \delta_2.$ Thus $\Vert x\Vert =\delta_1 \Vert x_0\Vert \leq \delta_2 \Vert x_0\Vert =\Vert y\Vert.$ Hence $(\mathbb{X}, \mathbb{X}^+, \Vert \cdot \Vert)$ is an ordered normed space.

Since $\mathbb{X}^+\subseteq \mathbb{X}_{x_0}^+$ and $\mathbb{X}_{x_0}^+$ is proper, we conclude that $\mathbb{X}^+$ is also proper. Finally, let $\lbrace \delta_n x_0\rbrace$ be a sequence in $\mathbb{X}^+$ such that $\delta_n x_0 \to y \in \mathbb{X}.$ Then $\vert \delta_n -\delta_m \vert =\frac{\Vert \delta_n x_0-\delta_m x_0\Vert}{\Vert x_0\Vert}$ and $\lbrace \delta_n x_0\rbrace$ is a Cauchy sequence in $\mathbb{X}$ so that $\lbrace \delta_n \rbrace$ is a Cauchy sequence in $\mathbb{R}.$ Since $\delta_n \geq 0$ for all $n\in \mathbb{N}$ and $\mathbb{R}$ is complete, there exists unique $\delta \in \mathbb{R}$ such that $\delta_n \to \delta.$ In this, case $\delta_n x_0\to \delta x_0=y.$ Thus $y\in \mathbb{X}^+$ and consequently $\mathbb{X}^+$ is norm-closed. 

\item[(2)] Let $\alpha=\Vert x_0\Vert.$ Let $x,y\in \mathbb{X}$ such that $0\leq x\leq y$ in $\mathbb{X}_{x_0}^+.$ Thus $y-x\in \mathbb{X}_{x_0}^+$ and $y=(y-x)+x.$ Then $x^*(y-x)\geq \alpha \Vert y-x\Vert$ and $x^*(x)\geq \alpha \Vert x\Vert$  so that $\Vert y\Vert =\Vert (y-x) +x\Vert = \lbrace y^*((y-x) +x): y^*\in \mathbb{X}^*, \Vert y^*\Vert=1\rbrace \geq \frac{x^*}{\Vert x^*\Vert}((y-x) +x)\geq \frac{x^*}{\Vert x^*\Vert}(x)\geq \frac{\alpha}{\Vert x^*\Vert} \Vert x\Vert=\frac{\Vert x_0\Vert}{\Vert x_0\Vert} \Vert x\Vert=\Vert x\Vert.$ 
\end{enumerate} 
\end{proof}

Let $\mathbb{X}$ be a normed space and $x_0 \in \mathbb{X}\setminus\lbrace 0\rbrace.$ If $x^*$ is the bounded linear functional induced by $x_0$ as in Proposition \ref{5}, then by $f_1\sim_{x_0}f_2,$ we simply mean that $f_1\sim_{x^*}f_2.$ 

The following result is regarded as Jordan decomposition for WBV-functions in an ordered normed space.

\begin{theorem}\label{12}
Let $\mathbb{X}$ be a normed space. If $f:[a,b]\to \mathbb{X}$ is a function of weakly bounded variation, then for each $x_0\in \mathbb{X}\setminus \lbrace 0\rbrace,$ there exists an extensible cone $\mathbb{X}_{x_0}^+$ in $\mathbb{X}$ and $f_{x_{0,1}},f_{x_{0,2}}:[a,b]\to \mathbb{X}$ increasing functions in $(\mathbb{X},X_{x_0}^+)$ such that $f\sim_{x_0}(f_{x_{0,1}}-f_{x_{0,2}}).$ In other words, givan a function of weakly bounded variation $f:[a,b]\to \mathbb{X},$ there exist functions of weakly bounded variation $f_{x_{0,1}},f_{x_{0,2}}:[a,b]\to \mathbb{X}$  such that $f\sim_{x_0}(f_{x_{0,1}}-f_{x_{0,2}}).$
\end{theorem}

\begin{proof}
Let $f:[a,b]\to \mathbb{X}$ be a function of bounded variation and $x_0\in \mathbb{X}\setminus \lbrace 0\rbrace.$ Also, let $x^*$ be a bounded linear functional constructed by $x_0$ as in the Proposition \ref{5}. Consider $x^* \circ f:[a,b]\to \mathbb{R}.$ Then $x^* \circ f$ is also a function of bounded variation. By Jordan decomposition theorem, there exists increasing functions $f_1, f_2:[a,b]\to \mathbb{R}$ such that $x^*\circ f=f_1-f_2.$ Put $g_1(t)=f_1(t)x_0$ and $g_2(t)=f_2(t)x_0$ for all $t\in [a,b]$ so that $g_1, g_2:[a,b]\to X$ increasing functions in $(X,X_{x_0}^+).$ Then $f_{x_{0,1}}=\frac{g_1}{\Vert x_0\Vert^2}, f_{x_{0,2}}=\frac{g_2}{\Vert x_0\Vert^2}$ are also increasing functions in $(X,X_{x_0}^+).$ Next, we have 
\begin{eqnarray*}
x^*\circ ( f-f_{x_{0,1}}+f_{x_{0,2}})(t) &=& x^*(f(t))-x^*(f_{x_{0,1}}(t))+x^*(f_{x_{0,2}}(t)) \\
&=& x^*(f(t))-\frac{x^*(g_1(t))}{\Vert x_0\Vert^2}+\frac{x^*(g_2(t))}{\Vert x_0\Vert^2} \\
&=& x^*(f(t))-\frac{f_1(t)}{\Vert x_0\Vert^2}x^*(x_0)+\frac{f_2(t)}{\Vert x_0\Vert^2}x^*(x_0) \\
&=& x^*(f(t))-\frac{f_1(t)}{\Vert x_0\Vert^2}\Vert x_0\Vert^2+\frac{f_2(t)}{\Vert x_0\Vert^2}\Vert x_0\Vert^2 \\
\end{eqnarray*}
\begin{eqnarray*}
\hspace{1.5 cm} &=& x^*\circ f-(f_1-f_2)(t) \\
&=& x^*\circ f(t)-x^*\circ f(t)=0 
\end{eqnarray*}

for all $t\in [a,b].$ Thus $h_0\circ (f-f_{x_{0,1}}+f_{x_{0,2}})=0$ so that $f\sim_{x_0}(f_{x_{0,1}}-f_{x_{0,2}}).$ 

Our proof is completed if we can show the functions $f_{x_{0,1}}$ and $f_{x_{0,2}}$ are of weakly bounded variation. Consider a partition $P=\lbrace a=t_0<t_1<t_2<\cdots <t_{n-1}<t_n\rbrace$ of $[a,b].$ Since $f_{x_{0,1}}$ is increasing, we get that $f_{x_{0,1}}(t_i)-f_{x_{0,1}}(t_{i-1})\in \mathbb{X}_+$ for $i=1,2,\cdots, n.$ Put $\beta=\frac{1}{\alpha}.$ Then $\Vert f_{x_{0,1}}(t_i)-f_{x_{0,1}}(t_{i-1})\Vert \leq \beta x^*(f_{x_{0,1}}(t_i)-f_{x_{0,1}}(t_{i-1})).$ For fix $y^*\in \mathbb{X}^*,$ we get that      
\begin{eqnarray*}
\displaystyle \sum_{i=1}^n \vert (y^*\circ f)(t_i)-(y^*\circ f)(t_{i-1})\vert
&=& \displaystyle \sum_{i=1}^n \vert y^* (f(t_i)-f(t_{i-1}))\vert \\
&\leq& \Vert y^*\Vert \displaystyle \sum_{i=1}^n \Vert f(t_i)-f(t_{i-1}))\Vert \\
&\leq& \Vert y^*\Vert \beta \displaystyle \sum_{i=1}^n x^*(f(t_i)-f(t_{i-1}))\\
&=& \Vert y^*\Vert \beta x^*\left (\displaystyle \sum_{i=1}^n (f(t_i)-f(t_{i-1}))\right)\\
&=& \Vert y^*\Vert \beta x^* (f(b)-f(a)).
\end{eqnarray*}
\end{proof}

The sum $\displaystyle \sum_{i=1}^n \vert (y^*\circ f)(t_i)-(y^*\circ f)(t_{i-1})\vert$ is bounded above by a fix number $\Vert y^*\Vert \beta x^* (f(b)-f(a))$ for any partition $P=\lbrace a=t_0<t_1<t_2<\cdots <t_{n-1}<t_n\rbrace$ of $[a,b].$ Therefore by order completeness or Dedekind completeness property of real number system, we conclude that $y^*\circ f_{x_{0,1}}$ is of bounded variation and consequently $f_{x_{0,1}}$ is of weakly bounded variation. Similarily, we can show that  $f_{x_{0,2}}$ is of weakly bounded variation.

\vspace{0.2 cm}

The result \ref{12} is a Jordan decomposition theorem for WBV-functions in an ordered normed space as a function of weakly bounded variation is not exactly equal to difference of two functions of weakly bounded variation in an extensible cone induced by a non-zero point but it is equal in the bounded functional sense $i.e.$ after taking composition of these functions by a bounded linear functional induced by the non-zero point. 

Next result provides a weak relation for an increasing function in $(\mathbb{X},\mathbb{X}_{x_0}^+)$ in the terms of bounded linear functional $x^*$ induced by $x_0.$

\begin{lemma}\label{001}
Let $\mathbb{X}$ be a normed space and $\mathbb{X}_{x_0}^+$ be an extensible cone induced by $x_0\in \mathbb{X}\setminus \lbrace 0\rbrace.$ If $f:[a,b]\to \mathbb{X}$ is an increasing function in $(\mathbb{X},\mathbb{X}_{x_0}^+)$ and $x^*$ be bounded linear functional induced by $x_0,$ then $x^*\circ f:[a,b]\to \mathbb{R}$ is increasing and satisfies that $f\sim_{x_0} x^*\circ f(t)\frac{x_0}{\Vert x_0\Vert^2}.$ 
\end{lemma}

\begin{proof}
$f$ is increasing function implies that $f$ is a function of weakly bounded variation. In particular, $x^*\circ f:[a,b]\to \mathbb{R}$ is a function of bounded variation. Consider corresponding variation function $\mathcal{V}_{x^*\circ f}$ for $x^*\circ f.$ By Theorem \ref{12}, there exists $f_{x_{0,1}}$ and $f_{x_{0,2}}$ increasing functions in $(X,X_{x_0}^+)$ such that $f\sim_{x_0}(f_{x_{0,1}}-f_{x_{0,2}})~i.e.$ $x^*\circ (f-f_{x_{0,1}}+f_{x_{0,2}})=0.$ In fact, we can choose $f_{x_{0,1}}(t)=\mathcal{V}_{x^*\circ f}(t)\frac{x_0}{\Vert x_0\Vert^2}$ and $f_{x_{0,2}}(t)=[\mathcal{V}_{x^*\circ f}(t)-x^*\circ f(t)]\frac{x_0}{{\Vert x_0\Vert^2}}$ so that $f\sim_{x_0}x^*\circ f(t)\frac{x_0}{\Vert x_0\Vert^2}.$ 

Next, let $t_2>t_1.$ Since $f$ is increasing in $(X,X_{x_0}^+),$ we have $f(t_2)-f(t_1)\in X_{x_0}^+.$ Then $x^*\circ f(t_2)-x^*\circ f(t_1)=x^*(f(t_2)-f(t_1))\geq \alpha \Vert f(t_2)-f(t_1)\Vert \geq 0$ so that $x^*\circ f(t_2)\geq x^*\circ f(t_1)~i.e.~x^*\circ f$ is increasing. 
\end{proof}

\begin{corollary}\label{002}
Let $\mathbb{X}$ be a normed space and $\mathbb{X}_{x_0}^+$ be an extensible cone induced by $x_0\in \mathbb{X}$ such that $\Vert x_0\Vert=1.$ If $f:[a,b]\to \mathbb{X}$ is an increasing function in $(\mathbb{X},\mathbb{X}_{x_0}^+)$ and $x^*$ be bounded linear functional induced by $x_0,$ then $x^*\circ f:[a,b]\to \mathbb{R}$ is increasing and satisfies that $f\sim_{x_0} x^*\circ f(t) x_0.$ 
\end{corollary}

Finally, in a normed space $\mathbb{X},$ for each $x_0\in \mathbb{X}\setminus \lbrace 0\rbrace,$ we construct a function of weakly bounded variation having a nice Jordan decomposition which satisfy weak relation in terms of the functional $x^*$ induced by $x_0.$

\begin{theorem}\label{003}
Let $\mathbb{X}$ be a normed space such that $\dim{\mathbb{X}}>1.$ For each $x_0\in \mathbb{X}\setminus \lbrace 0\rbrace,$ there exists $f:[a,b]\to \mathbb{X},$ a function of weakly bounded variation such that only possibility of satisfying $f\sim_{x_0}(f_{x_{0,1}}-f_{x_{0,2}}),$ where $f_{x_{0,1}}$ and $f_{x_{0,2}}$ increasing functions in $(\mathbb{X},\mathbb{X}_{x_0}^+),$ is that $f_{x_{0,1}} \sim_{x_0} f_{x_{0,2}}.$
\end{theorem}

\begin{proof}
Let $x_0\in \mathbb{X}\setminus \lbrace 0\rbrace.$ Then $x^*(x_0)=\Vert x_0\Vert^2\neq 0$ implies that range$(x^*)=\mathbb{R}.$ Since $\dim{\mathbb{X}}>1,$ there exists $y\in X\setminus \lbrace 0\rbrace$ such that $x^*(y)=0.$ Fix an incresaing function $\gamma:[a,b]\to \mathbb{R}.$ Put $f(t)=\gamma(t)y$ for all $t\in [a,b].$ Consider a partition $P=\lbrace a=t_0<t_1<t_2<\cdots <t_{n-1}<t_n\rbrace$ of $[a,b].$ For a fix $z^*\in \mathbb{X}^*,$ we get that $\displaystyle \sum_{i=1}^n \vert z^*\circ f(t_i)-z^*\circ f(t_{i-1})\vert = \displaystyle \sum_{i=1}^n \vert \gamma(t_i)  z^*(y)-\gamma(t_{i-i})  z^*(y)\vert =\vert z^*(y)\vert \displaystyle \sum_{i=1}^n \vert \gamma(t_i)-\gamma(t_{i-i})\vert=\vert z^*(y)\vert \displaystyle \sum_{i=1}^n  \gamma(t_i)-\gamma(t_{i-i})=\vert z^*(y)\vert (\gamma(b)-\gamma(a)).$ The sum $\displaystyle \sum_{i=1}^n \vert z^*\circ f(t_i)-z^*\circ f(t_{i-1})\vert$  is bounded above by a fix number $\vert z^*(y)\vert (\gamma(b)-\gamma(a))$ for any partition $P=\lbrace a=t_0<t_1<t_2<\cdots <t_{n-1}<t_n\rbrace$ of $[a,b].$ Therefore by order completeness or Dedekind completeness property of real number system, we conclude that $z^*\circ f$ is of bounded variation and consequently $f$ is of weakly bounded variation. By theorem \ref{12}, there exists $f_{x_{0,1}}, f_{x_{0,2}}:[a,b]\to \mathbb{X}$ increasing functions in $(X,X_{x_0}^+)$ such that $f\sim_{x_0}(f_{x_{0,1}} - f_{x_{0,2}})~i.e.~x^*\circ (f-f_{x_{0,1}}+f_{x_{0,2}})=0.$ Then 
\begin{eqnarray*}
0 &=& x^*\circ (f-f_{x_{0,1}}+f_{x_{0,2}})(t) \\
&=& x^*(f(t)-x^*(f_{x_{0,1}}(t)-f_{x_{0,2}}(t))\\
&=& x^*(\gamma(t)y)-x^*(f_{x_{0,1}}(t)-f_{x_{0,2}}(t))\\
&=& \gamma(t) x^*(y)-x^*(f_{x_{0,1}}(t)-f_{x_{0,2}}(t))\\
&=&\gamma(t) 0-x^*(f_{x_{0,1}}(t)-f_{x_{0,2}}(t))\\
&=& -x^*(f_{x_{0,1}}(t)-f_{x_{0,2}}(t))
\end{eqnarray*}
so that $x^*(f_{x_{0,1}}(t)-f_{x_{0,2}}(t))=0$ for all $t\in [a,b].$ Therefore $x^*\circ(f_{x_{0,1}}-f_{x_{0,2}})=0~i.e.~f_{x_{0,1}}\sim_{x_0}f_{x_{0,2}}.$
\end{proof}

\end{document}